\documentclass[11pt]{article}
\usepackage{enumerate}
\usepackage{amssymb,a4wide,latexsym,makeidx,epsfig}
\usepackage{amsthm}
\usepackage{dsfont}
\usepackage{amsmath}
\usepackage{lipsum}
\usepackage{enumerate}
\usepackage{mathrsfs}
\usepackage{xcolor}
\usepackage{tikz}
\usepackage{setspace}%设置行距
\usepackage{geometry}
\usepackage{appendix}
\usepackage{multirow}
\usepackage[colorlinks=true, linkcolor=black, anchorcolor=black, citecolor=black, urlcolor=black, CJKbookmarks=true]{hyperref}%urlcolor变成黑色
\usepackage{microtype}
\usepackage{colonequals}

\allowdisplaybreaks

\newtheorem{theorem}{Theorem}[section]
\newtheorem{remark}[theorem]{Remark}

\newtheorem{lemma}[theorem]{Lemma}

\newtheorem{proposition}[theorem]{Proposition}
\newtheorem{corollary}[theorem]{Corollary}

\newtheorem{question}[theorem]{Question}

\newtheorem{claim}{Claim}[section]

\newcommand{\rsat}{{\rm rsat}}
\newcommand{\rwsat}{{\rm rwsat}}
\newcommand{\sat}{{\rm sat}}
\newcommand{\wsat}{{\rm wsat}}
\newcommand{\ex}{{\rm ex}}

\begin{document}
\textwidth 150mm \textheight 225mm

\title{Weak rainbow saturation numbers of graphs
}
\author{
Xihe Li\footnote{School of Mathematical Sciences, University of Science and Technology of China, Hefei 230026, China.}~~~~~~
Jie Ma\footnotemark[1]~\footnote{Hefei National Laboratory, University of Science and Technology of China, Hefei 230088, China.}~~~~~~
Tianying Xie\footnotemark[1]
}
\date{}
\maketitle
\begin{center}
\begin{minipage}{120mm}
\vskip 0.3cm
\begin{center}
{\small {\bf Abstract}}
\end{center}
{\small For a fixed graph $H$, we say that an edge-colored graph $G$ is \emph{weakly $H$-rainbow saturated} if there exists an ordering $e_1, e_2, \ldots, e_m$ of $E\left(\overline{G}\right)$ such that, for any list $c_1, c_2, \ldots, c_m$ of pairwise distinct colors from $\mathbb{N}$, the non-edges $e_i$ in color $c_i$ can be added to $G$, one at a time, so that every added edge creates a new rainbow copy of $H$. The \emph{weak rainbow saturation number} of $H$, denoted by $\rwsat(n,H)$, is the minimum number of edges in a weakly $H$-rainbow saturated graph on $n$ vertices. In this paper, we show that for any non-empty graph $H$, the limit $\lim_{n\to \infty} \frac{\rwsat(n, H)}{n}$ exists. This answers a question of Behague, Johnston, Letzter, Morrison and Ogden [{\it SIAM J. Discrete Math.} (2023)]. We also provide lower and upper bounds on this limit, and in particular, we show that this limit is nonzero if and only if $H$ contains no pendant edges.
}
\end{minipage}
\end{center}

\section{Introduction}
\label{sec:introduction}

Typical extremal graph theory problems ask for the maximum or minimum value of parameters of graphs with certain properties. A classical example falling within this framework is the Tur\'{a}n problem which asks, for a fixed graph $H$, what is the maximum number of edges in an $H$-free\footnote{Given two graphs $G$ and $H$, we say that $G$ is $H$-free if $G$ contains no subgraph isomorphic to $H$.} graph on $n$ vertices. Another classical problem is the saturation problem which was initiated by Zykov \cite{Zyk} in the 1940s and first studied by Erd\H{o}s, Hajnal and Moon \cite{ErHM} in the 1960s. For a fixed graph $H$, a graph $G$ is called \emph{$H$-saturated} if $G$ is $H$-free but adding any non-edge to $G$ creates a copy of $H$. The \emph{saturation number} $\sat(n, H)$ is the smallest number of edges in an $H$-saturated graph on $n$ vertices. Erd\H{o}s, Hajnal and Moon \cite{ErHM} proved that $\sat(n, K_t)={n\choose 2}-{n-t+2 \choose 2}$, and this was generalized to the hypergraph setting by Bollob\'{a}s \cite{Bol965} using the well-known set-pairs inequality. A graph $G$ is called \emph{weakly $H$-saturated} if there exists an ordering $e_1, e_2, \ldots, e_m$ of the non-edges of $G$ such that for each $i\in [m]$, the graph $G_i\colonequals G+\{e_1, \ldots, e_i\}$ contains a copy of $H$ containing $e_i$ as an edge. The \emph{weak saturation number} $\wsat(n, H)$ is the smallest number of edges in a weakly $H$-saturated graph on $n$ vertices. In \cite{Bol968}, Bollob\'{a}s conjectured that $\wsat(n, K_t)=\sat(n, K_t)$. This conjecture was confirmed by Kalai \cite{Kal} using exterior algebra, and reproved by Alon \cite{Alon985JCTA} using the skewed version of the Bollob\'{a}s set-pairs inequality. Moreover, Alon \cite{Alon985JCTA} proved that the limit $\lim_{n\to \infty} \frac{\wsat(n, H)}{n}$ exists for every non-empty graph $H$. The hypergraph version of Alon's result was conjectured by Tuza \cite{Tuza992DM} in 1992 and proved by Shapira and Tyomkyn \cite{ShTy} very recently. In 1986, Tuza \cite{Tuza986,Tuza988} conjectured that the limit $\lim_{n\to \infty} \frac{\sat(n, H)}{n}$ exists for every graph $H$, and this conjecture still remains open; see \cite[Section~14]{CFFS} for more information.

The edge-coloring version of the saturation problem was raised by Hanson and Toft \cite{HaTo} in 1987. For a graph $G$, we refer to a mapping $c: E(G) \to \mathbb{N}$ as an \emph{edge-coloring} of $G$. A graph with an edge-coloring is called \emph{monochromatic} if all edges are colored the same. Hanson and Toft \cite{HaTo} focused on the saturation problem of monochromatic cliques. A graph with an edge-coloring is called \emph{rainbow} if all edges are colored differently. The study of rainbow colored graphs can be traced back to the Latin square decomposition problem initiated by Euler in the 1780s. In combinatorics, many classical problems can be transferred to the problem of finding certain rainbow substructures in edge-colored graphs, such as Ringel's conjecture \cite{MoPS21}, the Ryser-Brualdi-Stein conjecture \cite{Pok18} and the Caccetta-H\"{a}ggkvist conjecture \cite{AhDH}. In the last two decades, rainbow generalizations of Tur\'{a}n-type problems \cite{Jan23IJM,KMSV} and Ramsey-type problems \cite{FoGP,LiBW,SuVo} became an active research area. The rainbow generalization of saturation problems was first studied by Barrus, Ferrara, Vandenbussche and Wenger \cite{BFVW} in 2017. They considered saturation problems of rainbow subgraphs in an edge-colored host graph with a bounded number of colors.

For a fixed graph $H$, we say that an edge-colored graph $G$ is \emph{$H$-rainbow saturated} if $G$ does not contain a rainbow copy of $H$, but the addition of any non-edge in any color from $\mathbb{N}$ creates a rainbow copy of $H$. Gir\~{a}o, Lewis and Popielarz \cite{GiLP} defined the \emph{rainbow saturation number} of $H$, denoted by $\rsat(n,H)$, to be the minimum number of edges in an $H$-rainbow saturated graph on $n$ vertices. Gir\~{a}o, Lewis and Popielarz \cite{GiLP} conjectured that the rainbow saturation number of any non-empty graph is at most linear in $n$. Recently, Behague, Johnston, Letzter, Morrison and Ogden \cite{BJLMO} confirmed this conjecture. For more related works, we refer the interested reader to \cite{BuJR,CaMT,CHLT,Kor}. For a fixed graph $H$, we say that an edge-colored graph $G$ is \emph{weakly $H$-rainbow saturated} if there exists an ordering $e_1, e_2, \ldots, e_m$ of $E\left(\overline{G}\right)$ such that, for any list $c_1, c_2, \ldots, c_m$ of pairwise distinct colors from $\mathbb{N}$, the non-edges $e_i$ in color $c_i$ can be added to $G$, one at a time, so that every added edge creates a new rainbow copy of $H$. Behague et al. \cite{BJLMO} defined the \emph{weak rainbow saturation number} of $H$, denoted by $\rwsat(n,H)$, to be the minimum number of edges in a weakly $H$-rainbow saturated graph on $n$ vertices.\footnote{In the case that there exists no $H$-rainbow saturated (resp., weakly $H$-rainbow saturated) graph on $n$ vertices, we define $\rsat(n,H)\colonequals |E(K_n)|$ (resp., $\rwsat(n,H)\colonequals |E(K_n)|$).}

As pointed out in \cite[Section~6]{BJLMO}, in the definition of the weak rainbow saturation number, we require the collection of added edges to receive pairwise distinct colors, so in particular, we exclude the possibility that all added edges have the same color, in which case the previously added edges do not contribute to making new rainbow copies and the problem reduces to the standard rainbow saturation number. Moreover, note that $c_1, c_2, \ldots, c_m$ are colors from $\mathbb{N}$, so some of them might be used within the original edges of $G$. Furthermore, for a weakly $H$-rainbow saturated graph $G$, we do not require $G$ itself to be rainbow $H$-free.

By the definitions, we have $\wsat(n, H)\leq \rwsat(n, H)\leq \rsat(n, H)$. Hence, the above mentioned result of Behague et al. \cite{BJLMO} on $\rsat(n, H)$ implies that $\rwsat(n, H)=O(n)$ for any non-empty graph $H$. Extending the result of Alon~\cite{Alon985JCTA}, Behague et al. posed the following question.

\begin{question}[\cite{BJLMO}]\label{ques:rwsat}
For any non-empty graph $H$, does the limit $\lim_{n\to \infty} \frac{\rwsat(n, H)}{n}$ exist?
\end{question}

In this paper, we fully resolve this question by proving the following theorem.

\begin{theorem}\label{th:rwsat_limits}
For any non-empty graph $H$, the limit $\lim_{n\to \infty} \frac{\rwsat(n, H)}{n}$ exists.
\end{theorem}

In the special case that $H$ is a complete graph, Behague et al. \cite{BJLMO} proved that $\rwsat(n, K_t)\leq (t+2\sqrt{2t})n+c_t$ for $t\geq 3$, where $c_t$ is a constant depending only on $t$. They asked whether $\rwsat(n, K_t)\leq tn+O(1)$ holds for every integer $t\geq 3$ and sufficiently large $n$. This question was solved by Chakraborti, Hendrey, Lund and Tompkins \cite{CHLT} recently by showing that $\rwsat(n, K_t)\leq (t-1)n+O(1)$ holds for every integer $t\geq 3$. Our second result extends this result form complete graphs to general graphs. In particular, our result implies that the limit $\lim_{n\to \infty} \frac{\rwsat(n, H)}{n}$ is nonzero if and only if $H$ contains no pendant edges. An edge is \emph{pendant} if one of its endpoints has degree 1. For any graph $H$ and vertex $v\in V(H)$, let $N_H(v)$ be the neighborhood of $v$ in $H$, and let $d_H(v)\colonequals |N_H(v)|$ be the degree of $v$. Let $\delta(H)$ be the minimum degree of $H$ and let $\delta'(H)\colonequals \min\{d_H(v)\colon\, v\in V(H), d_H(v)\neq 0\}$.

\begin{theorem}\label{th:rwsat_limits0}
Let $H$ be a non-empty graph. Then the following statements hold.
\begin{itemize}
\item[{\rm (i)}] If $H$ contains a pendant edge, then $\lim_{n\to \infty} \frac{\rwsat(n, H)}{n}=0$.
\item[{\rm (ii)}] If $H$ contains no pendant edge, then $\frac{1}{2}\delta'(H)\leq \lim_{n\to \infty} \frac{\rwsat(n, H)}{n}\leq \delta'(H)$.
\end{itemize}
\end{theorem}

The remainder of this paper is organized as follows. In the next section, we introduce some additional terminology and notation, and prove some lemmas that will be used in our proofs of the main results. In Section~\ref{sec:proof}, we will complete our proof of Theorem~\ref{th:rwsat_limits}, and establish Theorem~\ref{th:rwsat_limits0} in a more precise form. Finally, we conclude the paper with some remarks and open problems in Section~\ref{sec:ch-conclu}.

\section{Preliminaries}
\label{sec:pre}

We begin with some additional terminology and notation. Given an edge-colored graph $G$ and an edge $e\in E(G)$, we use $c_G(e)$ to denote the color assigned on $e$. Given two disjoint vertex subsets $U, V\subseteq V(G)$, let $E_G(U, V)\colonequals \{uv\in E(G)\colon\, u\in U, v\in V\}$. If $U$ consists of a single vertex $u$, we simply write $E_G(\{u\}, V)$ as $E_G(u, V)$. The subscript $G$ in $c_G(e)$, $E_G(U, V)$, $N_G(v)$ and $d_G(v)$ will be omitted if $G$ is clear from the context. For a vertex subset $U\subseteq V(G)$, we use $G[U]$ to denote the \emph{edge-colored induced subgraph} of $G$, that is, $V(G[U])=U$, $E(G[U])=\{e\in E(G)\colon\, e\subseteq U\}$, and each edge in $G[U]$ receives the same color as it receives in $G$. For a vertex subset $V\subseteq V(G)$, let $G-V\colonequals G[V(G)\setminus V]$. Given a set $E$ of non-edges (resp., edges) of $G$, let $G+E$ (resp., $G-E$) be the graph obtained form $G$ by adding (resp., deleting) all the edges in $E$. If $E$ consists of a single edge $e$, we simply write $G+\{e\}$ and $G-\{e\}$ as $G+e$ and $G-e$, respectively. Given two vertex-disjoint graphs $G$ and $H$, we use $G\cup H$ to denote the \emph{disjoint union} of $G$ and $H$, that is, the graph with vertex set $V(G\cup H)=V(G)\cup V(H)$ and edge set $E(G\cup H)=E(G)\cup E(H)$. We use $nG$ to denote the disjoint union of $n$ copies of $G$.

Next, we state and prove several lemmas.

\begin{lemma}\label{le:turan_KST}
Let $n\geq m\geq 2$ be two positive integers. Then the following statements hold.
\begin{itemize}
\item[{\rm (i)}] For any constant $c$ with $0\leq c\leq \frac{n-3}{6}$, every $n$-vertex graph with at most $cn$ edges contains an independent set of size $\lceil\frac{n}{2c+1}\rceil$.

\item[{\rm (ii)}] Every subgraph of $K_{m,n}$ with at least $m(n-1)$ edges contains $K_{\lfloor\frac{m}{2}\rfloor,\lfloor\frac{n}{2}\rfloor}$ as a subgraph.
\end{itemize}
\end{lemma}

\begin{proof} (i) Let $G$ be an $n$-vertex graph with $|E(G)|\leq cn$. Then the number of edges in the complement $\overline{G}$ of $G$ satisfies
$$\left|E\left(\overline{G}\right)\right|\geq {n\choose 2}-cn= \frac{n-2c-1}{n}\cdot \frac{n^2}{2}= \left(1-\frac{1}{n/(2c+1)}\right)\frac{n^2}{2} > \left(1-\frac{1}{\lceil n/(2c+1)\rceil -1}\right)\frac{n^2}{2}.$$
By Tur\'{a}n's Theorem, $\overline{G}$ contains a complete subgraph of order $\lceil\frac{n}{2c+1}\rceil$, and thus $G$ contains an independent set of size $\lceil\frac{n}{2c+1}\rceil$.

(ii) Let $G$ be a subgraph of $K_{m,n}$ with bipartition $(A,B)$ such that $|A|=m$, $|B|=n$ and $|E(G)|\geq m(n-1)$. Let $A'\colonequals \{v\in A\colon\, d(v)\leq n-2\}$ and $A''=A\setminus A'$. Then $|A'|\leq \frac{1}{2}(mn-m(n-1))=\frac{m}{2}$ and $|A''|\geq m-|A'|\geq \frac{m}{2}$. Take an arbitrary subset $A^{\ast}\subset A''$ with $|A^{\ast}|=\lfloor\frac{m}{2}\rfloor$. Note that every vertex $v\in A^{\ast}$ has at least $n-1$ neighbors in $B$. Hence, there exists a subset $B^{\ast}\subseteq B$ with $|B^{\ast}|\geq n-|A^{\ast}|\geq n-\frac{m}{2}\geq n-\frac{n}{2}=\frac{n}{2}$ such that $A^{\ast}\cup B^{\ast}$ induces a complete bipartite subgraph of $G$. The result follows.
\end{proof}

Given a family $\mathscr{F}$ of graphs, let $\ex(n, \mathscr{F})$ be the \emph{Tur\'{a}n number} of $\mathscr{F}$, that is, the maximum number of edges in an $n$-vertex graph that contains no members of $\mathscr{F}$. For any graph $H$, let $f(H)$ be the smallest integer $n$ such that for each $N\in \{n-1, n\}$ we have $\ex(N, \mathscr{H})\leq {N\choose 2}-2N-2$, where $\mathscr{H}\colonequals \{H-\{u, v\}\colon\, uv\in E(H)\}$. Note that $|V(H)|-1\leq f(H)\leq 5|V(H)|$ (for the upper bound, see the proof of Corollary~\ref{cor:add} below).

\begin{lemma}\label{le:gen_add}
Let $H$ be a graph with $E(H)\neq \emptyset$, $F$ be a complete graph of order $f(H)+2$, and $u,v$ be two distinct vertices of $F$. We color the edges of $F$ such that $F-\{u,v\}$ is rainbow, and all the edges between $\{u,v\}$ and $V(F)\setminus \{u,v\}$ form a rainbow copy of $K_{2, n-2}$. Then no matter what color is assigned on $uv$, there is a rainbow copy of $H$ containing the edge $uv$.
\end{lemma}

\begin{proof} The result holds trivially if $|V(H)|\leq 2$, so we may assume that $|V(H)|\geq 3$ in the following argument. Since $E(\{u,v\}, V(F)\setminus \{u,v\})$ forms a rainbow subgraph, we can remove at most one vertex from $V(F)\setminus \{u, v\}$ to get a subset $V\subseteq V(F)\setminus \{u, v\}$ such that $f(H)-1\leq |V|\leq f(H)$ and $E(\{u,v\},V)$ contains no edges of the color $c(uv)$. Let $F'$ be the subgraph of $F[V]$ consisting of all its edges using colors from $\mathbb{N}\setminus (\{c(e)\colon\, e\in E(\{u, v\}, V)\}\cup \{c(uv)\})$. Since $F-\{u,v\}$ is rainbow and $f(H)-1\leq |V|\leq f(H)$, we have $|E(F')|\geq {|V|\choose 2}-2|V|-1> \ex(|V|, \mathscr{H})$. This implies that $F'$ contains a copy $H^{\ast}$ of $H-\{x,y\}$ for some edge $xy\in E(H)$. Note that $H^{\ast}$ is rainbow and contains no edges using colors from $\{c(e)\colon\, e\in E(\{u, v\}, V)\}\cup \{c(uv)\}$. This implies that $F[V(H^{\ast})\cup \{u,v\}]$ contains a rainbow copy of $H$ containing the edge $uv$. The result follows.
\end{proof}

\begin{corollary}\label{cor:add}
Let $H$ be a graph with $E(H)\neq \emptyset$, $F$ be a copy of $K_n$ with $n\geq 5|V(H)|$, and $u,v$ be two distinct vertices of $F$. We color the edges of $F$ such that $F-\{u,v\}$ is rainbow, and all the edges between $\{u,v\}$ and $V(F)\setminus \{u,v\}$ form a rainbow copy of $K_{2, n-2}$. Then no matter what color is assigned on $uv$, there is a rainbow copy of $H$ containing the edge $uv$.
\end{corollary}

\begin{proof} We first show that $f(H)\leq 5|V(H)|$. To this end, let $G$ be a graph on $N\geq 5|V(H)|-1$ vertices with $|E(G)|={N\choose 2}-2N-2$. Then $|E\left(\overline{G}\right)|=2N+2$. Applying Lemma~\ref{le:turan_KST} (i) with $c=\frac{2N+2}{N}$ to $\overline{G}$, we can find an independent set $U$ of $\overline{G}$ with $|U|\geq \frac{N}{2c+1}\geq |V(H)|-2$. This implies that $G[U]$ is a complete subgraph on at least $|V(H)|-2$ vertices, so $G$ contains some graph $H^{\ast}\in \mathscr{H}$. This implies that $f(H)\leq 5|V(H)|$. Let $V\subseteq V(F)\setminus \{u,v\}$ with $|V|=f(H)$. The result follows by applying Lemma~\ref{le:gen_add} to $F[V\cup \{u,v\}]$.
\end{proof}

Recall that for a weakly $H$-rainbow saturated graph $G$, we do not require $G$ itself to be rainbow $H$-free. We have the following result on weakly rainbow saturated graphs.

\begin{lemma}\label{le:rwsat_property}
For any graph $H$, integer $n$, and weakly $H$-rainbow saturated graph $G$ on $n$ vertices, we can recolor the edges of $G$ such that the resulting edge-colored graph $G'$ is rainbow and $G'$ is still weakly $H$-rainbow saturated.
\end{lemma}

\begin{proof} Since $G$ is a weakly $H$-rainbow saturated graph, there exists an ordering $e_1, e_2, \ldots, e_m$ of $E\left(\overline{G}\right)$ such that, for any list $c_1, c_2, \ldots, c_m$ of pairwise distinct colors from $\mathbb{N}$, the non-edges $e_i$ in color $c_i$ can be added to $G$, one at a time, so that every added edge creates a new rainbow copy of $H$. Now we consider the rainbow graph $G'$. For an arbitrarily fixed list $c'_1, c'_2, \ldots, c'_m$ of pairwise distinct colors from $\mathbb{N}$, we wish to show that the non-edges $e_i$ in color $c'_i$ can be added to $G'$, one at a time, so that every added edge creates a new rainbow copy of $H$.

Let $G'_0\colonequals G'$, and for each $i\in [m]$ let $G'_i\colonequals G'+\{e_1, \ldots, e_i\}$ with $c(e_{k})=c'_k$ for every $k\in [i]$. Suppose for some $j\in [m]$, the non-edges $e_i$ (for each $i<j$) in color $c'_i$ can be added to $G'_{i-1}$ so that every added edge creates a new rainbow copy of $H$, but adding $e_j$ in color $c'_j$ to $G'_{j-1}$ does not create any new rainbow copy of $H$. Since $G$ is weakly $H$-rainbow saturated, the addition of $e_j$ creates at least one copy of $H$. Let $\mathcal{A}$ be the set of all underlying copies\footnote{For an edge-colored graph $F$, the underlying copy of $F$ is an uncolored graph consisting of the vertex set $V(F)$ and edge set $E(F)$.} of $H$ in $G+\{e_1, \ldots, e_j\}$ containing $e_j$. Then for any $A\in \mathcal{A}$, $A$ is a copy of $H$ but $A$ is not a rainbow subgraph of $G'_j$. Since $G'$ is rainbow and $c'_1, \ldots, c'_j$ are pairwise distinct, there exists at most one edge $e'_i$ in $G'$ with $c_{G'}(e'_i)=c'_i$ for each $i\in [j]$. For each $i\in [j]$, let $\mathcal{A}_{i}\colonequals \{A\in  \mathcal{A}\colon\, \mbox{$G'$ contains a unique edge $e'_i$ with $c_{G'}(e'_i)=c'_i$ and $e_i, e'_i\in E(A)$}\}$. Let $i_1, \ldots, i_t$ be all the indices such that $\mathcal{A}_{i_{\ell}}\neq \emptyset$ for each $\ell \in [t]$.  Note that $\mathcal{A}_{i_1}, \ldots, \mathcal{A}_{i_t}$ form a partition of $\mathcal{A}$. Then for any list $c_1, c_2, \ldots, c_m$ of pairwise distinct colors from $\mathbb{N}$ with $c_{i_{\ell}}=c_G(e'_{i_{\ell}})$ for each $\ell \in [t]$, we cannot add the non-edges $e_i$ in color $c_i$ to $G$, one at a time, so that every added edge creates a new rainbow copy of $H$. This contradiction completes the proof of Lemma~\ref{le:rwsat_property}.
\end{proof}

We shall also use the following version of Fekete's Subadditive Lemma.

\begin{lemma}[\cite{FuRu}]\label{le:Fekete}
Let $c$ and $t$ be two positive constants. For any sequence $\left\{a_{n}\right\}_{n\in \mathbb{N}}$ with $a_{m+n}\leq a_m+a_n+c$ for every $m,n\geq t$, the limit $\lim_{n\to \infty}\frac{a_{n}}{n}$ exists.
\end{lemma}

\section{Proofs of Theorems~\ref{th:rwsat_limits} and \ref{th:rwsat_limits0}}
\label{sec:proof}

We first present our proof of Theorem~\ref{th:rwsat_limits}. Our proof is inspired by the work of Alon in \cite{Alon985JCTA}.
\vspace{0.2cm}

\noindent {\bf Proof of Theorem~\ref{th:rwsat_limits}}~
Let $t=\max\{\lceil c(H)\rceil, |V(H)|, 3\}$, where $c(H)$ is a constant such that $\rsat(n, H)\leq c(H)n$ (guaranteed by the result of \cite{BJLMO}). Then $\rwsat(n, H)\leq \rsat(n, H)\leq tn$. We shall show that for every $m_1, m_2\geq t^{10}$,
\begin{equation}\label{eq:subadd-1}
\rwsat(m_1+m_2, H)\leq \rwsat(m_1, H) + \rwsat(m_2, H) + t^{14}.
\end{equation}

For each $i\in [2]$, let $G_i$ be a weakly $H$-rainbow saturated graph on $m_i\geq t^{12}$ vertices with $|E(G_i)|=\rwsat(m_i, H)$. By Lemma~\ref{le:rwsat_property}, we may further assume that $G_1$ and $G_2$ are two disjoint rainbow graphs and they have no common colors, i.e., $G_1\cup G_2$ is rainbow. For each $i\in [2]$, let $X_i\colonequals \{v\in V(G_i)\colon\, d_{G_i}(v)\geq \frac{m_i}{4}\}$. Note that $|X_i|\leq \frac{2|E(G_i)|}{m_i/4}\leq \frac{2tm_i}{m_i/4}=8t$ for each $i\in [2]$. By Lemma~\ref{le:turan_KST} (i), for each $i\in [2]$, $G_i$ contains an independent set of size at least $\frac{m_i}{2t+1} > t^6+8t$, so we may assume that $A_i\subseteq V(G_i)\setminus X_i$ is an independent set of $G_i$ with $|A_i|=t^6$. Let $G$ be the $(m_1+m_2)$-vertex graph obtained from $G_1\cup G_2$ by adding all edges between $X_1\cup A_1$ and $X_2\cup A_2$, and we color the new edges such that $G$ is rainbow. Note that $|E(G)|=|E(G_1)|+|E(G_2)|+|X_1\cup A_1||X_2\cup A_2|\leq \rwsat(m_1, H) + \rwsat(m_2, H) + t^{14}.$ In order to prove Inequality (\ref{eq:subadd-1}), it suffices to show that $G$ is weakly $H$-rainbow saturated.

Let $a$, $b$ and $s$ be the number of non-edges of $G_1$, $G_2$ and $G$, respectively. Then $s=a+b+m_1m_2-|X_1\cup A_1||X_2\cup A_2|$. We shall show that there exists an ordering $e_1, e_2, \ldots, e_s$ of non-edges of $G$ such that, for any list $c_1, c_2, \ldots, c_s$ of pairwise distinct colors from $\mathbb{N}$, the non-edges $e_i$ in color $c_i$ can be added to $G$, one at a time, so that every added edge creates a new rainbow copy of $H$. Since $G_1$ and $G_2$ are weakly $H$-rainbow saturated, there exists an ordering $e_1, e_2, \ldots, e_a$ of non-edges of $G_1$ and an ordering $e_{a+1}, e_{a+2}, \ldots, e_{a+b}$ of non-edges of $G_2$ such that, the non-edges $e_i$ in color $c_i$ can be added to $G$, one at a time, so that every added edge creates a new rainbow copy of $H$. Let $G^{(1)}=G+\{e_1, e_2, \ldots, e_{a+b}\}$ with $c(e_i)=c_i$ for each $i\in [a+b]$.

We next consider the non-edges between $V(G_1)$ and $V(G_2)$. For each $i\in [2]$, let $B_i\colonequals \{v\in V(G_i)\setminus (X_i\cup A_i)\colon\, |N_{G_i}(v)\cap A_i|\geq t^5\}$ and $C_i\colonequals V(G_i)\setminus (X_i\cup A_i\cup B_i)$. Note that for each $i\in [2]$, we have $|B_i|\leq \frac{|E(G_i)|}{t^5}\leq \frac{tm_i}{t^5}=\frac{m_i}{t^4}$ and $|C_i|\geq m_i-|X_i|-|A_i|-|B_i|\geq \frac{2m_i}{3}$. Roughly speaking, we will consider the remaining non-edges in the following ordering: $E_{\overline{G}}(C_1\cup A_1, C_2\cup A_2)$, $E_{\overline{G}}(B_1, X_2\cup A_2\cup C_2)\cup E_{\overline{G}}(B_2, X_1\cup A_1\cup C_1)$, $E_{\overline{G}}(B_1, B_2)$, $E_{\overline{G}}(C_1, X_2)\cup E_{\overline{G}}(C_2, X_1)$; see Figure~\ref{fig:rwsat}. For convenience, we introduce one more notion. Assume that $G^{\ast}$ is the edge-colored graph obtained from $G$ by adding certain non-edges $e_1, e_2, \ldots, e_{\ell}$ with $c(e_{i})=c_{i}$ for each $i\in [\ell]$. For a subset $E^{\ast}$ of non-edges of $G^{\ast}$, we say that $E^{\ast}$ \emph{is nice to} $G^{\ast}$ if there exists an ordering $e_{\ell+1}, e_{\ell+2}, \ldots, e_{\ell+|E^{\ast}|}$ of the non-edges in $E^{\ast}$ such that, the non-edges $e_{\ell+i}$ in color $c_{\ell+i}$ can be added to $G^{\ast}$, one at a time, so that every added edge creates a new rainbow copy of $H$.

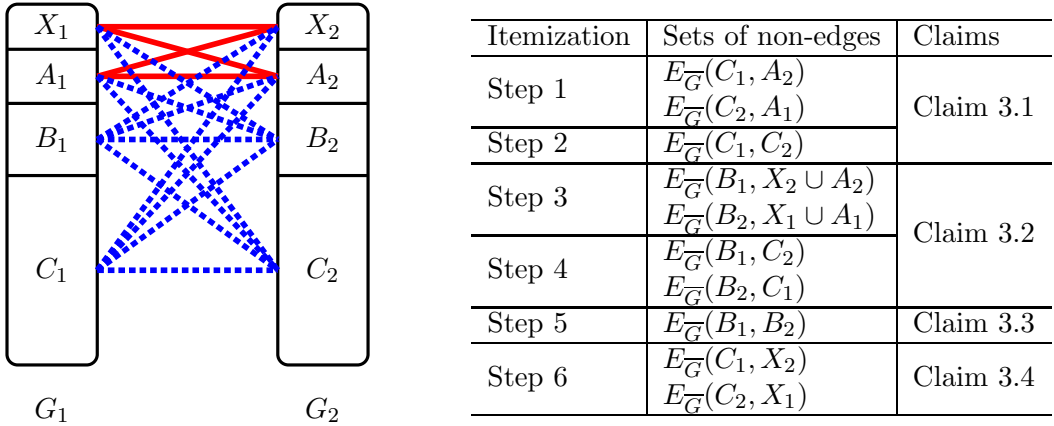
\begin{figure}[htbp]
	%% Figure
	\begin{minipage}[p]{0.45\textwidth}
		\centering
		\begin{tikzpicture}[scale=0.06,auto,swap]
		
		\draw[black, rounded corners, very thick] (0,0) rectangle (20,80);
		\draw[very thick] (0,70) -- (20,70); \draw[very thick] (0,58) -- (20,58); \draw[very thick] (0,42) -- (20,42);
		\draw (10,75) node {$X_1$}; \draw (10,64) node {$A_1$}; \draw (10,50) node {$B_1$}; \draw (10,21) node {$C_1$}; \draw (10,-10) node {$G_1$};
		
		\draw[black, rounded corners, very thick] (60,0) rectangle (80,80);
		\draw[very thick] (60,70) -- (80,70); \draw[very thick] (60,58) -- (80,58); \draw[very thick] (60,42) -- (80,42);
		\draw (70,75) node {$X_2$}; \draw (70,64) node {$A_2$}; \draw (70,50) node {$B_2$}; \draw (70,21) node {$C_2$}; \draw (70,-10) node {$G_2$};
		
		\draw[line width=0.75mm, red, opacity=0.5] (20,75) -- (60,75); \draw[line width=0.75mm, red, opacity=0.5] (20,75) -- (60,64);
		\draw[line width=0.75mm, red, opacity=0.5] (20,64) -- (60,75); \draw[line width=0.75mm, red, opacity=0.5] (20,64) -- (60,64);
		
		\draw[dash pattern=on 2.5pt off 1.5pt, line width=0.75mm, blue, opacity=0.4] (20,75) -- (60,50);
		\draw[dash pattern=on 2.5pt off 1.5pt, line width=0.75mm, blue, opacity=0.4] (20,75) -- (60,21);
		\draw[dash pattern=on 2.5pt off 1.5pt, line width=0.75mm, blue, opacity=0.4] (20,64) -- (60,50);
		\draw[dash pattern=on 2.5pt off 1.5pt, line width=0.75mm, blue, opacity=0.4] (20,64) -- (60,21);
		\draw[dash pattern=on 2.5pt off 1.5pt, line width=0.75mm, blue, opacity=0.4] (20,50) -- (60,75);
		\draw[dash pattern=on 2.5pt off 1.5pt, line width=0.75mm, blue, opacity=0.4] (20,50) -- (60,64);
		\draw[dash pattern=on 2.5pt off 1.5pt, line width=0.75mm, blue, opacity=0.4] (20,50) -- (60,50);
		\draw[dash pattern=on 2.5pt off 1.5pt, line width=0.75mm, blue, opacity=0.4] (20,50) -- (60,21);
		\draw[dash pattern=on 2.5pt off 1.5pt, line width=0.75mm, blue, opacity=0.4] (20,21) -- (60,75);
		\draw[dash pattern=on 2.5pt off 1.5pt, line width=0.75mm, blue, opacity=0.4] (20,21) -- (60,64);
		\draw[dash pattern=on 2.5pt off 1.5pt, line width=0.75mm, blue, opacity=0.4] (20,21) -- (60,50);
		\draw[dash pattern=on 2.5pt off 1.5pt, line width=0.75mm, blue, opacity=0.4] (20,21) -- (60,21);
		
		\end{tikzpicture}
	\end{minipage}
	%%Table
	\begin{minipage}[p]{0.55\textwidth}
		\centering
		\begin{tabular}{l|l|l}
			\cline{1-3}
			Itemization & Sets of non-edges & Claims \\  \cline{1-3}
			\multirow{2}{*}{Step 1} & $E_{\overline{G}}(C_1, A_2)$ & \multirow{3}{*}{Claim~\ref{cl:lim-1}} \\
			& $E_{\overline{G}}(C_2, A_1)$ & \\ \cline{1-2}
			Step 2 & $E_{\overline{G}}(C_1, C_2)$ & \\ \cline{1-3}
			\multirow{2}{*}{Step 3} & $E_{\overline{G}}(B_1, X_2\cup A_2)$ & \multirow{4}{*}{Claim~\ref{cl:lim-2}}\\
			& $E_{\overline{G}}(B_2, X_1\cup A_1)$ & \\ \cline{1-2}
			\multirow{2}{*}{Step 4} & $E_{\overline{G}}(B_1, C_2)$ & \\
			& $E_{\overline{G}}(B_2, C_1)$ & \\ \cline{1-3}
			Step 5 & $E_{\overline{G}}(B_1, B_2)$ & Claim~\ref{cl:lim-3}\\ \cline{1-3}
			\multirow{2}{*}{Step 6} & $E_{\overline{G}}(C_1, X_2)$ & \multirow{2}{*}{Claim~\ref{cl:lim-4}}\\
			& $E_{\overline{G}}(C_2, X_1)$ & \\ \cline{1-3}
		\end{tabular}
	\end{minipage}
	\caption{An illustration of the non-edges (represented by dashed lines) of $G^{(1)}$, and the rough ordering of the remaining non-edges that we wish to add to $G^{(1)}$.}
	\label{fig:rwsat}
\end{figure}

\begin{claim}\label{cl:lim-1} $E_{\overline{G}}(C_1\cup A_1, C_2\cup A_2)$ is nice to $G^{(1)}$.
\end{claim}

\begin{proof} For an arbitrarily fixed vertex $u\in C_1$, there exists a set $S\subseteq A_1$ with $|S|\geq |A_1|-t^5\geq t^5$ such that $G_1$ contains no edges between $u$ and $S$. Note that in $G^{(1)}$, the subset $S\cup \{u\}$ induces a rainbow complete subgraph with colors from $\{c_1, c_2, \ldots, c_{a}\}$, and the edges between $S$ and $A_2$ forms a rainbow complete bipartite subgraph. Let $G'$ be the bipartite subgraph of $G$ with bipartition $(S, A_2)$ and $E(G')=\{v_1v_2\colon\, v_1\in S, v_2\in A_2, \mbox{$c(v_1v_2)$ is not a color on edges between $u$ and $S$}\}$. Since $G$ is rainbow, we have $|E(G')|\geq |S||A_2|-|S|$. By Lemma~\ref{le:turan_KST} (ii), there exist subsets $S'\subseteq S$ and $A'\subseteq A_2$ with $|S'|=t^4$ and $|A'|=t^5$ such that $G'[S'\cup A']$ is a complete bipartite subgraph. This implies that the edges between $S'$ and $A'\cup \{u\}$ form a rainbow complete bipartite subgraph. For any vertex $x\in A'$, by Corollary~\ref{cor:add} (with $F=G^{(1)}[S'\cup \{u, x\}]+ux$), the addition of the non-edge $ux$ in any color to $G^{(1)}$ creates a new rainbow copy of $H$. By symmetry, for any ordering $e_{a+b+1}, \ldots, e_{a+b+t^5}$ of the non-edges between $u$ and $A'$, the non-edges $e_{a+b+i}$ in color $c_{a+b+i}$ can be added to $G^{(1)}$, one at a time, so that every added edge creates a new rainbow copy of $H$. Let $G^{(2,1)}=G^{(1)}+\{e_{a+b+1}, \ldots, e_{a+b+t^5}\}$ with $c(e_{a+b+i})=c_{a+b+i}$ for each $i\in \left[t^5\right]$.

For any vertex $y\in A_2\setminus A'$, note that $G^{(2,1)}[A'\cup \{u,y\}]$ is a rainbow subgraph with colors from $\{c_1, c_2, \ldots, c_{a+b+t^5}\}$. Thus adding the non-edge $uy$ in any color to $G^{(2,1)}$ creates a new rainbow copy of $H$. By symmetry, this in fact shows that there exists an ordering $e_{a+b+1}, \ldots, e_{a+b+t^6|C_1\cup C_2|}$ of the non-edges in $E_{\overline{G}}(C_1, A_2)\cup E_{\overline{G}}(C_2, A_1)$ such that, the non-edges $e_{a+b+i}$ in color $c_{a+b+i}$ can be added to $G^{(1)}$, one at a time, so that every added edge creates a new rainbow copy of $H$. Let $G^{(2,2)}=G^{(1)}+\{e_{a+b+1}, \ldots, e_{a+b+t^6|C_1\cup C_2|}\}$ with $c(e_{a+b+i})=c_{a+b+i}$ for each $i\in \left[t^6|C_1\cup C_2|\right]$.

Let $v_1v_2$ be an arbitrarily fixed non-edge with $v_1\in C_1$ and $v_2\in C_2$. Note that there exists a set $S''\subseteq A_1$ with $|S''|\geq |A_1|-t^5\geq t^5$ such that $G_1$ contains no edges between $v_1$ and $S''$. Then $G^{(2,2)}[S''\cup \{v_1,v_2\}]$ is a rainbow subgraph with colors from $\{c_1, c_2, \ldots, c_{s}\}$. Thus adding the non-edge $v_1v_2$ in any color to $G^{(2,2)}$ creates a new rainbow copy of $H$. Since $v_1v_2$ is chosen arbitrarily, we know that there exists an ordering $e_{a+b+t^6|C_1\cup C_2|+1}, \ldots, e_{a+b+t^6|C_1\cup C_2|+|C_1||C_2|}$ of the non-edges in $E_{\overline{G}}(C_1,C_2)$ such that, the non-edges $e_{a+b+t^6|C_1\cup C_2|+i}$ in color $c_{a+b+t^6|C_1\cup C_2|+i}$ can be added to $G^{(2,2)}$, one at a time, so that every added edge creates a new rainbow copy of $H$. This completes the proof of Claim~\ref{cl:lim-1}.
\end{proof}

Let $q_1=|E_{\overline{G}}(C_1\cup A_1, C_2\cup A_2)|=t^6|C_1\cup C_2|+|C_1||C_2|$. By Claim~\ref{cl:lim-1}, there exists an ordering $e_{a+b+1}, \ldots, e_{a+b+q_1}$ of the non-edges in $E_{\overline{G}}(C_1\cup A_1, C_2\cup A_2)$ such that, the non-edges $e_{a+b+i}$ in color $c_{a+b+i}$ can be added to $G^{(1)}$, one at a time, so that every added edge creates a new rainbow copy of $H$. Let $G^{(2)}=G^{(1)}+\{e_{a+b+1}, \ldots, e_{a+b+q_1}\}$ with $c(e_{a+b+i})=c_{a+b+i}$ for each $i\in [q_1]$.

\begin{claim}\label{cl:lim-2} $E_{\overline{G}}(B_1, X_2\cup A_2\cup C_2)\cup E_{\overline{G}}(B_2, X_1\cup A_1\cup C_1)$ is nice to $G^{(2)}$.
\end{claim}

\begin{proof} Let $(v, z)$ be an arbitrarily fixed pair of vertices with $v\in B_1$ and $z\in X_2\cup A_2$. Let $T\subseteq N_{G_1}(v)\cap A_1$ with $|T|=t^5$. Note that in $G^{(2)}+vz$, the vertex set $T\cup\{v, z\}$ induces a complete subgraph. Moreover, the edges between $T$ and $\{v, z\}$ forms a rainbow $K_{2,t^5}$, and $T$ induces a rainbow complete subgraph with colors from $\{c_1, c_2, \ldots, c_{a}\}$. By Corollary~\ref{cor:add} (with $F=G^{(2)}[T\cup \{v, z\}]+vz$), after adding $vz$ in any color, there is a new rainbow copy of $H$. By symmetry, the same statement holds for every pair of vertices $(v, z)$ with $(v,z)\in B_1\times (X_2\cup A_2)$ or $(v,z)\in B_2\times (X_1\cup A_1)$. Let $q=|E_{\overline{G}}(B_1, X_2\cup A_2)\cup E_{\overline{G}}(B_2, X_1\cup A_1)|=|B_1||X_2\cup A_2|+|B_2||X_1\cup A_1|$. Then there exists an ordering $e_{a+b+q_1+1}, \ldots,$ $e_{a+b+q_1+q}$ of the non-edges in $E_{\overline{G}}(B_1, X_2\cup A_2)\cup E_{\overline{G}}(B_2, X_1\cup A_1)$ such that, the non-edges $e_{a+b+q_1+i}$ in color $c_{a+b+q_1+i}$ can be added to $G^{(2)}$, one at a time, so that every added edge creates a new rainbow copy of $H$. Let $G^{(3,1)}=G^{(2)}+\{e_{a+b+q_1+1}, \ldots, e_{a+b+q_1+q}\}$ with $c(e_{a+b+q_1+i})=c_{a+b+q_1+i}$ for each $i\in [q]$.

Let $(u, x)$ be an arbitrarily fixed pair of vertices with $u \in C_1$ and $x\in B_2$. Note that there exists a set $S\subseteq A_1$ with $|S|\geq |A_1|-t^5\geq t^5$ such that $G_1$ contains no edges between $u$ and $S$. Then $G^{(3,1)}[S\cup \{u,x\}]$ is a rainbow subgraph with colors from $\{c_1, c_2, \ldots, c_{a+b+q_1+q}\}$. Thus adding the non-edge $ux$ in any color to $G^{(3,1)}$ creates a new rainbow copy of $H$. By symmetry, the same statement holds for every pair of vertices $(u, x)$ with $(u,x)\in C_1\times B_2$ or $(u,x)\in C_2\times B_1$. This completes the proof of Claim~\ref{cl:lim-2}.
\end{proof}

Let $q_2=|E_{\overline{G}}(B_1, X_2\cup A_2\cup C_2)\cup E_{\overline{G}}(B_2, X_1\cup A_1\cup C_1)|$. By Claim~\ref{cl:lim-2}, there exists an ordering $e_{a+b+q_1+1}, \ldots,$ $e_{a+b+q_1+q_2}$ of the non-edges in $E_{\overline{G}}(B_1, X_2\cup A_2\cup C_2)\cup E_{\overline{G}}(B_2, X_1\cup A_1\cup C_1)$ such that, the non-edges $e_{a+b+q_1+i}$ in color $c_{a+b+q_1+i}$ can be added to $G^{(2)}$, one at a time, so that every added edge creates a new rainbow copy of $H$. Let $G^{(3)}=G^{(2)}+\{e_{a+b+q_1+1}, \ldots, e_{a+b+q_1+q_2}\}$ with $c(e_{a+b+q_1+i})=c_{a+b+q_1+i}$ for each $i\in [q_2]$.

\begin{claim}\label{cl:lim-3} $E_{\overline{G}}(B_1, B_2)$ is nice to $G^{(3)}$.
\end{claim}

\begin{proof} Let $w$ be an arbitrarily fixed vertex of $B_1$. Since $w\notin X_1$, there exists a subset $C'\subseteq C_1$ such that $G$ contains no edges between $w$ and $C'$ and $|C'|\geq |C_1|-\frac{m_1}{4}\geq \frac{2m_1}{3}-\frac{m_1}{4}> \frac{m_1}{3}$. Note that $|E(G_1[C'])|\leq |E(G_1)|\leq tm_1 \leq 3t |C'|$. Applying Lemma~\ref{le:turan_KST} (i) (with $c=3t$) to $G_1[C']$, we have that $G_1$ contains an independent set $C''\subseteq C'$ of size at least $\frac{|C'|}{6t+1}\geq t^4$. Note that for any vertex $w'\in B_2$, the subset $C''\cup \{w,w'\}$ induces a rainbow subgraph of $G^{(4)}$. Thus after adding $ww'$ in any color, there is a new rainbow copy of $H$. Since $w$ and $w'$ are chosen arbitrarily, the result follows.
\end{proof}

Let $q_3=|B_1||B_2|$. By Claim~\ref{cl:lim-3}, there exists an ordering $e_{a+b+q_1+q_2+1}, \ldots, e_{a+b+q_1+q_2+q_3}$ of the non-edges between $B_1$ and $B_2$ such that, the non-edges $e_{a+b+q_1+q_2+i}$ in color $c_{a+b+q_1+q_2+i}$ can be added to $G^{(3)}$, one at a time, so that every added edge creates a new rainbow copy of $H$. Let $G^{(4)}=G^{(3)}+\{e_{a+b+q_1+q_2+1}, \ldots, e_{a+b+q_1+q_2+q_3}\}$ with $c(e_{a+b+q_1+q_2+i})=c_{a+b+q_1+q_2+i}$ for each $i\in [q_3]$.

\begin{claim}\label{cl:lim-4} $E_{\overline{G}}(C_1, X_2)\cup E_{\overline{G}}(C_2, X_1)$ is nice to $G^{(4)}$.
\end{claim}

\begin{proof} Let $C'\subseteq C_1$ be the set of vertices $v$ such that there exists a vertex $x\in A_1$ with $vx\notin E(G)$ and $c_{G^{(4)}}(vx)\in \{c(e)\colon\, e\in E_G(A_1, X_2)\}$. Let $C''=C_1\setminus C'$. Since $c_1, c_2, \ldots, c_s$ are pairwise distinct colors, we have $|C'|\leq |A_1||X_2|\leq 8t^7$ and $|C''|=|C_1|-|C'|\geq \frac{2m_1}{3}-8t^7\geq \frac{m_1}{3}$.

Let $uy$ be an arbitrarily fixed non-edges with $u\in C''$ and $y\in X_2$. Note that there exists a set $S\subseteq A_1$ with $|S|\geq |A_1|-t^5\geq t^5$ such that $G_1$ contains no edges between $u$ and $S$. Note that in $G^{(4)}+uy$, the vertex set $S\cup\{u, y\}$ induces a complete subgraph. Moreover, the edges between $S$ and $\{u, y\}$ forms a rainbow $K_{2,t^5}$, and $S$ induces a rainbow complete subgraph with colors from $\{c_1, c_2, \ldots, c_{a}\}$. By Corollary~\ref{cor:add} (with $F=G^{(4)}[S\cup \{u, y\}]+uy$), after adding $uy$ in any color, there is a new rainbow copy of $H$. Since $uy$ is chosen arbitrarily, we know that there exists an ordering $e_{a+b+q_1+q_2+q_3+1}, \ldots, e_{a+b+q_1+q_2+q_3+|C''||X_2|}$ of the non-edges in $E_{\overline{G}}(C'',X_2)$ such that, the non-edges $e_{a+b+q_1+q_2+q_3+i}$ in color $c_{a+b+q_1+q_2+q_3+i}$ can be added to $G^{(4)}$, one at a time, so that every added edge creates a new rainbow copy of $H$. Let $G^{(5,1)}=G^{(4)}+\{e_{a+b+q_1+q_2+q_3+1}, \ldots, e_{a+b+q_1+q_2+q_3+|C''||X_2|}\}$ with $c(e_{a+b+q_1+q_2+q_3+i})=c_{a+b+q_1+q_2+q_3+i}$ for each $i\in \left[|C''||X_2|\right]$.

Let $w$ be an arbitrarily fixed vertex of $C'$. Since $w\notin X_1$, there exists a subset $C'''\subseteq C''$ such that $G$ contains no edges between $w$ and $C'''$ and $|C'''|\geq |C''|-\frac{m_1}{4}\geq \frac{m_1}{3}-\frac{m_1}{4}\geq \frac{m_1}{12}$. Note that $|E(G_1[C'''])|\leq |E(G_1)|\leq tm_1 \leq 12t |C'''|$. Applying Lemma~\ref{le:turan_KST} (i) (with $c=12t$) to $G_1[C''']$, we have that $G_1$ contains an independent set $C^{\ast}\subseteq C'''$ of size at least $\frac{|C'''|}{24t+1}\geq t^3$. Note that for any vertex $w'\in X_2$, the subset $C^{\ast}\cup \{w,w'\}$ induces a rainbow subgraph of $G^{(4)}$. Thus after adding $ww'$ in any color, there is a new rainbow copy of $H$. Since $w$ and $w'$ are chosen arbitrarily, this holds for every non-edge between $C'$ and $X_2$. By symmetry, this in fact implies that there exists an ordering $e_{a+b+q_1+q_2+q_3+1}, \ldots, e_{s}$ of the non-edges in $E_{\overline{G}}(C_1, X_2)\cup E_{\overline{G}}(C_2, X_1)$ such that, the non-edges $e_{a+b+q_1+q_2+q_3+i}$ in color $c_{a+b+q_1+q_2+q_3+i}$ can be added to $G^{(4)}$, one at a time, so that every added edge creates a new rainbow copy of $H$. The proof of Claim~\ref{cl:lim-4} is complete.
\end{proof}

By Claim~\ref{cl:lim-4}, $G$ is weakly $H$-rainbow saturated, and thus Inequality (\ref{eq:subadd-1}) holds. By Lemma~\ref{le:Fekete}, the limit $\lim_{n\to \infty} \frac{\rwsat(n, H)}{n}$ exists. This completes the proof of Theorem~\ref{th:rwsat_limits}.
\hfill$\square$
\vspace{0.2cm}

We next present our proof of Theorem~\ref{th:rwsat_limits0} in the following more precise form, which generalizes a result of Faudree, Gould and Jacobson \cite{FaGJ} on weak saturation numbers, and a result of Chakraborti, Hendrey, Lund and Tompkins \cite{CHLT} on weak rainbow saturation numbers of complete graphs. Recall that for any graph $H$, $f(H)$ is the smallest integer $n$ such that for each $N\in \{n-1, n\}$ we have $\ex(N, \mathscr{H})\leq {N\choose 2}-2N-2$, where $\mathscr{H}\colonequals \{H-\{u, v\}\colon\, uv\in E(H)\}$.

\begin{theorem}\label{th:rwsat_bounds}
Let $H$ be a non-empty graph. Then the following statements hold.
\begin{itemize}
\item[{\rm (i)}] If $H$ contains a pendant edge and $n>f(H)+1$, then $\rwsat(n, H)\leq {f(H)+1 \choose 2}$.
\item[{\rm (ii)}] If $H$ contains no pendant edges and $n>f(H)+\delta'(H)$, then
$$\frac{1}{2}\delta'(H)n \leq\rwsat(n, H)\leq \delta'(H)(n-f(H)-\delta'(H))+{f(H)+\delta'(H) \choose 2}.$$
\end{itemize}
\end{theorem}

\begin{proof}
(i) We shall show that if $H$ contains a pendant edge and $n> f(H)+1\geq |V(H)|$, then $\rwsat(n,H)\leq {f(H)+1\choose 2}$. Let $G$ be an $n$-vertex graph consisting of a rainbow clique of order $f(H)+1$ and $n-f(H)-1$ isolated vertices. It suffices to show that $G$ is weakly $H$-rainbow saturated. Let $c_1, c_2, \ldots, c_{m}$ be an arbitrarily fixed list of pairwise distinct colors from $\mathbb{N}$, where $m={n\choose 2}-{f(H)+1\choose 2}$. Let $U\subset V(G)$ be the set of $n-f(H)-1$ isolated vertices, $V=V(G)\setminus U$, and $s=|U||V|=(n-f(H)-1)(f(H)+1)$.

For any pair of vertices $(u,v)\in U\times V$ and any color $c^{\ast}\in \mathbb{N}$, there exists a subset $V'\subseteq V$ with $v\in V'$ and $|V'|\geq |V|-1= f(H)\geq |V(H)|-1$ such that $G[V']$ is a rainbow clique and contains no edge of color $c^{\ast}$. Thus the addition of $uv$ in color $c^{\ast}$ creates a new rainbow copy of $H$ (with $uv$ being the pendant edge). By symmetry, this implies that for any ordering $e_1, e_2, \ldots, e_{s}$ of the non-edges between $U$ and $V$, the non-edges $e_i$ in color $c_i$ can be added to $G$, one at a time, so that every added edge creates a new rainbow copy of $H$. Let $G'=G+\{e_1, e_2, \ldots, e_{s}\}$ with $c(e_i)=c_i$ for each $i\in [s]$.

For any non-edge $u_1u_2$ within $U$, we consider the subgraph $F=G'[V^{\ast}\cup \{u_1, u_2\}]+u_1u_2$, where $V^{\ast}$ is an arbitrary subset of $V$ with $|V^{\ast}|=f(H)$. Note that $F$ satisfies the hypothesis of Lemma~\ref{le:gen_add}. Thus the addition of $u_1u_2$ in any color creates a new rainbow copy of $H$. By symmetry, this implies that for any ordering $e_{s+1}, e_{s+2}, \ldots, e_{m}$ of the non-edges within $U$, the non-edges $e_{s+i}$ in color $c_{s+i}$ can be added to $G'$, one at a time, so that every added edge creates a new rainbow copy of $H$. Therefore, $G$ is weakly $H$-rainbow saturated, and thus $\rwsat(n,H)\leq |E(G)|= {f(H)\choose 2}$.

(ii) For the lower bound, let $G$ be a weakly $H$-rainbow saturated graph on $n$ vertices. Then there exists an ordering $e_1, e_2, \ldots, e_m$ of $E\left(\overline{G}\right)$ such that, for any list $c_1, c_2, \ldots, c_m$ of pairwise distinct colors from $\mathbb{N}$, the non-edges $e_i$ in color $c_i$ can be added to $G$, one at a time, so that every added edge creates a new rainbow copy of $H$. Since $H$ contains no pendant edges and $E(H)\neq \emptyset$, we have $d_H(u)=0$ or $d_H(u)\geq 2$ for every vertex $u\in V(H)$, i.e., $\delta'(H)\geq 2$. This implies that $G$ contains no isolated vertices.

Suppose for a contradiction that $|E(G)|<\frac{1}{2}\delta'(H)n$. Then $G$ contains a vertex $v$ with $d\colonequals d_G(v)\leq \delta'(H)-1$. Let $e_{i_1}, \ldots, e_{i_{n-1-d}}$ be all the non-edges of $G$ containing $v$ as an end-vertex, where $i_1<\cdots<i_{n-1-d}$. Let $G_{i_1-1}=G+\{e_1, \ldots, e_{i_1-1}\}$ and $G_{i_1}=G_{i_1-1}+e_{i_1}$. Then $d_{G_{i_1-1}}(v)=d_G(v)\leq \delta'(H)-1$ and $d_{G_{i_1}}(v)=d_G(v)+1\leq \delta'(H)$. If $c_{i_1}$ is a color from the set of colors on edges incident with $v$ in $G$, then the addition of $e_{i_1}$ in color $c_{i_1}$ to $G'$ does not create any new rainbow copy of $H$, a contradiction. This implies that $\rwsat(n, H)\geq \frac{1}{2}\delta'(H)n$.

For the upper bound, consider the following construction. Let $A$, $B$ and $C$ be three pairwise disjoint sets of vertices with $|A|=\delta'(H)$, $|B|=f(H)$ and $|C|=n-f(H)-\delta'(H)$. Let $G'$ be a rainbow graph on vertex set $A\cup B\cup C$ whose edge set consists of all edges within $A\cup B$ and all edges between $A$ and $C$. Then $|E(G')|=\delta'(H)(n-f(H)-\delta'(H))+{f(H)+\delta'(H)\choose 2}.$ It suffices to show that $G'$ is weakly $H$-rainbow saturated.

Let $c'_1, c'_2, \ldots, c'_{t}$ be an arbitrarily fixed list of pairwise distinct colors from $\mathbb{N}$, where $t={n\choose 2}-|E(G')|$. Let $xy$ be an arbitrary non-edge with $x\in B$ and $y\in C$, and let $G''$ be the edge-colored graph obtained from $G'$ by adding $xy$ in any color. Since $G'$ is rainbow, there exists a subset $X\subseteq A\cup B$ with $x\in X$ and $|X|\geq |A\cup B|-1\geq |V(H)|-1$ such that $G''[X\cup \{y\}]$ is rainbow, $d_{G''[X\cup \{y\}]}(y)\geq \delta'(H)$, and $d_{G''[X\cup \{y\}]}(z)=|X|$ for every $z\in X$. This implies that adding $xy$ in any color creates a new rainbow copy of $H$. By symmetry, this implies that for any ordering $e'_1, e'_2, \ldots, e'_{|B||C|}$ of the non-edges between $B$ and $C$, the non-edges $e'_i$ in color $c'_i$ can be added to $G'$, one at a time, so that every added edge creates a new rainbow copy of $H$. Let $G'''=G'+\{e'_1, e'_2, \ldots, e'_{|B||C|}\}$ with $c(e'_i)=c'_i$ for each $i\in [|B||C|]$. Next, we consider the remaining non-edges, i.e., non-edges within $C$. Let $w, w'$ be two distinct vertices of $C$, and let $F=G'''[B\cup \{w, w'\}]+ww'$. Then $F$ satisfies the hypothesis of Lemma~\ref{le:gen_add}. Thus the addition of $ww'$ in any color creates a new rainbow copy of $H$. By symmetry, this implies that for any ordering $e'_{|B||C|+1}, e'_{|B||C|+2}, \ldots, e'_{t}$ of the non-edges within $C$, the non-edges $e'_{|B||C|+i}$ in color $c'_{|B||C|+i}$ can be added to $G'''$, one at a time, so that every added edge creates a new rainbow copy of $H$. Therefore, $G'$ is weakly $H$-rainbow saturated, and thus $\rwsat(n,H)\leq |E(G')|=\delta'(H)(n-f(H)-\delta'(H))+{f(H)+\delta'(H)\choose 2}$. This completes the proof.
\end{proof}

\begin{remark}\label{le:bounds}
In the case when $H$ is a complete graph $K_{r}$ $(r\geq 3)$, the upper bound given by Theorem~\ref{th:rwsat_bounds} {\rm (ii)} can be improved to $\rwsat(n, K_{r})\leq(r-1)(n-r)+{r\choose 2}$. Indeed, in this case, when we construct the graph $G'$, we may choose $B$ to be a single vertex. Moreover, when we add the non-edge $ww'$, we can find a rainbow copy of $H$ within $A\cup \{w,w'\}$ $($so we can avoid the use of Lemma~\ref{le:gen_add}$)$. This upper bound on $\rwsat(n, K_{r})$ was first obtained by Chakraborti, Hendrey, Lund and Tompkins \cite{CHLT}.
\end{remark}

\section{Concluding remarks}
\label{sec:ch-conclu}

In this paper, we prove that the limit $\lim_{n\to \infty} \frac{\rwsat(n, H)}{n}$ exists for any non-empty graph $H$. We also show that this limit is nonzero if and only if $H$ contains no pendant edges by proving that if $H$ contains no pendant edges and $n>f(H)+\delta'(H)$, then
\begin{equation}\label{eq:bound}
\frac{1}{2}\delta'(H)n \leq\rwsat(n, H)\leq \delta'(H)(n-f(H)-\delta'(H))+{f(H)+\delta'(H) \choose 2},
\end{equation}
where $\delta'(H)\colonequals \min\{d_H(v)\colon\, v\in V(H), d_H(v)\neq 0\}$.

For sufficiently large $n$, the lower bound in Inequality (\ref{eq:bound}) cannot be improved to $cn$ for any $c>\frac{1}{2}(\delta'(H)+1)$. To see this, let $H$ be the graph obtained from $2K_{t}$ ($t\geq 3$) by adding a single edge between the two copies of $K_{t}$. Note that $\delta'(H)=t-1$. For sufficiently large $n$, we write $n=\lfloor\frac{n}{t+1}\rfloor(t+1)+r$, where $0\leq r \leq t$. Let $G$ be a rainbow copy of $(rK_{t+2})\cup ((\lfloor\frac{n}{t+1}\rfloor-r)K_{t+1})$. It is easy to check that $G$ is weakly $H$-rainbow saturated. Thus $\rwsat(n, H)\leq |E(G)|=\frac{t}{2}n+\Theta(1)=\frac{1}{2}(\delta'(H)+1)n+\Theta(1)$. For sufficiently large $n$, the upper bound in Inequality (\ref{eq:bound}) cannot be improved to $c'n$ for any $c'<\delta'(H)-1$. For example, when $H=K_t$ ($t\geq 3$), we have $\rwsat(n, H)\geq \wsat(n, H)={n\choose 2}-{n-t+2\choose 2}=(t-2)n-\Theta(1)=(\delta'(H)-1)n-\Theta(1)$ (see \cite{Alon985JCTA,Kal}). Given this, we pose the following two questions.

\begin{question}\label{ques:lower}
Let $H$ be a non-empty graph containing no pendant edges. Is it true that $\rwsat(n, H)\geq \left(\frac{\delta'(H)+1}{2}-o(1)\right)n$?
\end{question}

\begin{question}\label{ques:complete}
For any integer $t\geq 3$, does there exist a constant $c_t$ such that $\rwsat(n, K_t)= (t-1)n+c_t$?
\end{question}

It is also natural to ask for what graphs $H$, it holds $\rwsat(n, H)\leq cn$ for some $c<\delta'(H)$ and sufficiently large integers $n$. Our Theorem~\ref{th:rwsat_bounds} (i) implies that this is the case when $\delta'(H)=1$. We can also show that this holds for a large family of graphs $H$ with $\delta'(H)=2$ (including all cycles of length at least 5). Let $\mathscr{F}$ be the family of graphs $H$ containing an edge $uv$ with $d_H(u)=d_H(v)=2$ such that $uv$ is the middle edge of an induced subgraph $P_4$ in $H$. Note that for any $H\in \mathscr{F}$, we have $\delta'(H)\in \{1, 2\}$.

\begin{proposition}\label{prop:degree2}
For any graph $H\in \mathscr{F}$ with $\delta'(H)=2$, there exists a constant $c_H$ such that $\rwsat(n, H)\leq \frac{3}{2}n+c_H$.
\end{proposition}

\begin{proof} Let $h=|V(H)|$ and let $P$ be the induced $P_4$ of $H$ as described in the definition of $\mathscr{F}$. We may assume that $n\geq h+3$ since we can choose $c_H$ to be a constant greater than ${h+2\choose 2}$. Let $k=\lfloor\frac{n-h-1}{2}\rfloor$ and $t=n-2k$, so $h+1\leq t\leq h+2$. Let $G$ be a rainbow graph on $n$ vertices with $V(G)=\{v_1, \ldots, v_t, x_1, \ldots, x_k, y_1, \ldots, y_k\}$ and $|E(H)|=\{v_iv_j\colon\, 1\leq i<j\leq t\}\cup \{v_1x_i, v_1y_i, x_iy_i \colon\, i\in [k]\}$. Note that $|E(G)|={t\choose 2}+3k=\frac{3}{2}n+c_H$ for some constant $c_H$. It suffices to show that $G$ is weakly $H$-rainbow saturated.

Let $c_1, c_2, \ldots, c_{m}$ be an arbitrarily fixed list of pairwise distinct colors from $\mathbb{N}$ , where $m={n\choose 2}-|E(G)|$. Let $V=\{v_1, \ldots, v_t\}$ and $U=\{x_1, \ldots, x_k, y_1, \ldots, y_k\}$. We first consider the non-edges between $V$ and $U$. By symmetry, we only consider the addition of $v_2x_1$ in some color $c^{\ast}\in \{c_1, c_2, \ldots, c_{m}\}$. Since $G$ is rainbow, we can find a subset $V'\subseteq V\setminus \{v_1, v_2\}$ with $|V'|\geq |V|-3\geq h-2$ such that $G[V'\cup \{v_1, v_2\}]-v_1v_2$ contains no edges of color $c^{\ast}$. If $c_G(x_1v_1)\neq c^{\ast}$, then the addition of $v_2x_1$ in color $c^{\ast}$ creates a rainbow copy of $H$ in which $v_2x_1v_1v_i$ plays the role of the path $P$, where $v_i$ is a vertex of $V'$. If $c_G(x_1v_1)= c^{\ast}$, then the addition of $v_2x_1$ in color $c^{\ast}$ creates a rainbow copy of $H$ in which $v_2x_1y_1v_1$ plays the role of the path $P$. Therefore, for any ordering $e_1, \ldots, e_{2k(t-1)}$ of the non-edges between $V$ and $U$, the non-edges $e_i$ in color $c_i$ can be added to $G$, one at a time, so that every added edge creates a new rainbow copy of $H$. Let $G'=G+\{e_1, e_2, \ldots, e_{2k(t-1)}\}$ with $c(e_i)=c_i$ for each $i\in [2k(t-1)]$. We next consider the non-edges within $U$. By symmetry, we only consider the addition of $x_1x_2$ in some color $c^{\ast\ast}\in \{c_{2k(t-1)+1}, \ldots, c_{m}\}$. Since $G$ is rainbow and $c_1, c_2, \ldots, c_{m}$ are pairwise distinct, we can find a subset $V''\subseteq V\setminus \{v_2, v_3\}$ with $|V''|\geq |V|-5\geq h-4$ such that $G'[V''\cup \{v_2, v_3\}]-v_2v_3$ contains no edges of colors from $\{c^{\ast\ast}, c_{G'}(v_2x_1), c_{G'}(v_3x_2)\}$. Then the addition of $x_1x_2$ in color $c^{\ast\ast}$ creates a rainbow copy of $H$ in which $v_2x_1x_2v_3$ plays the role of the path $P$. Therefore, for any ordering $e_{2k(t-1)+1}, \ldots, e_{m}$ of the non-edges within $U$, the non-edges $e_{2k(t-1)+i}$ in color $c_{2k(t-1)+i}$ can be added to $G'$, one at a time, so that every added edge creates a new rainbow copy of $H$. This completes the proof.
\end{proof}

Note that Proposition~\ref{prop:degree2} also implies that for all cycles $C_{\ell}$ with $\ell\geq 5$, we have $\rwsat(n, C_{\ell})\leq \frac{3}{2}n+c_{\ell}$, where $c_{\ell}$ is a constant only depending on $\ell$. This statement also holds for $C_4$.\footnote{To see this, consider a rainbow copy of the graph $G_1$ or $G_2$ on $n$ vertices defined as follows. For an odd integer $n$, let $G_1$ be the graph with vertex set $\{u, v_1, v_2, \ldots, v_{n-1}\}$ and edge set $\{uv_i\colon\, 1\leq i\leq n-1\}\cup \{v_{2i-1}v_{2i}\colon\, 1\leq i\leq \frac{n-1}{2}\}$. For an even integer $n$, let $G_2$ be the graph with vertex set $\{u, x, y, z, v_1, v_2, \ldots, v_{n-4}\}$ and edge set $\{xy, xz, yz, ux, uy, uz\}\cup \{uv_i\colon\, 1\leq i\leq n-4\}\cup \{v_{2i-1}v_{2i}\colon\, 1\leq i\leq \frac{n-4}{2}\}$.}
In light of this, we propose the following question.

\begin{question}\label{ques:cycles}
For any integer $\ell \geq 4$, there is a constant $c_{\ell}$ such that $\rwsat(n, C_{\ell})= \frac{3}{2}n+c_{\ell}$.
\end{question}

\section*{Acknowledgement}

Jie Ma was supported by National Key R\&D Program of China 2023YFA1010201, National Natural Science Foundation of China grant 12125106, and Anhui Initiative in Quantum Information Technologies grant AHY150200.
Xihe Li was supported by the Fundamental Research Funds for the Central Universities.

\begin{spacing}{0.8}

\end{spacing}

\vspace{0.5cm}

\noindent{\it E-mail address}: lxhdhr@163.com
\vspace{0.1cm}

\noindent{\it E-mail address}: jiema@ustc.edu.cn
\vspace{0.1cm}

\noindent{\it E-mail address}: xiety@ustc.edu.cn

\end{document}